\documentclass[12pt]{amsart}

\usepackage{amsmath}
\usepackage{amssymb}  
\usepackage{latexsym} 
\usepackage{comment}
\usepackage{url}
\usepackage{fullpage,url,amssymb,amsmath,amsthm,amsfonts,mathrsfs}
\usepackage[usenames,dvipsnames]{color}
\usepackage[pagebackref = true, colorlinks = true, linkcolor = blue, citecolor = Green]{hyperref}
\usepackage[alphabetic,lite,nobysame]{amsrefs}
\usepackage{enumitem}
\usepackage{amscd}   
\usepackage[all, cmtip]{xy} 
\usepackage{xfrac}

\DeclareFontEncoding{OT2}{}{} 
\newcommand{\textcyr}[1]{%
 {\fontencoding{OT2}\fontfamily{wncyr}\fontseries{m}\fontshape{n}\selectfont #1}}

\usepackage[all]{xy}
\usepackage{fullpage}
\newcommand{\Sha}{{\mbox{\textcyr{Sh}}}}

\usepackage{color} 

\def\act#1#2%
  {\mathop{}%
   \mathopen{\vphantom{#2}}^{#1}%
   \kern-3\scriptspace%
   #2}

\newcommand{\Z}{{\mathbb Z}}
\newcommand{\Q}{{\mathbb Q}}
\newcommand{\R}{{\mathbb R}}

\newcommand{\F}{{\mathbb F}}

\newcommand{\PP}{{\mathbb P}}

\newcommand{\kbar}{{\overline{k}}}
\newcommand{\Cbar}{{\overline{C}}}

\newcommand{\calJ}{{\mathcal J}}

\newcommand{\fm}{{\mathfrak m}}

\newcommand{\To}{\longrightarrow}

\DeclareMathOperator{\End}{End}
\DeclareMathOperator{\Hom}{Hom}

\DeclareMathOperator{\disc}{disc}

\DeclareMathOperator{\Gal}{Gal}

\DeclareMathOperator{\Res}{Res}

\DeclareMathOperator{\Pic}{Pic}

\DeclareMathOperator{\Jac}{Jac}
\DeclareMathOperator{\HH}{H}

\DeclareMathOperator{\GL}{GL}
\DeclareMathOperator{\SL}{SL}
\DeclareMathOperator{\Sp}{Sp}

\newcommand{\res}{\operatorname{res}}

\newtheorem{Theorem}{Theorem}
\newtheorem{Lemma}[Theorem]{Lemma}
\newtheorem{Proposition}[Theorem]{Proposition}

\newtheorem{Example}[Theorem]{Example}
\newtheorem{Remark}[Theorem]{Remark}

\numberwithin{equation}{section}

\begin{document}

\title{Most binary forms come from a pencil of quadrics}

\author{Brendan Creutz}
\address{School of Mathematics and Statistics, University of Canterbury, Private Bag 4800, Christchurch 8140, New Zealand}
\email{brendan.creutz@canterbury.ac.nz}
\urladdr{http://www.math.canterbury.ac.nz/\~{}bcreutz}

\begin{abstract}
A pair of symmetric bilinear forms $A$ and $B$ determine a binary form $f(x,y) := \disc(Ax-By)$. We prove that the question of whether a given binary form can be written in this way as a discriminant form generically satisfies a local-global principle and deduce from this that most binary forms over $\Q$ are discriminant forms. This is related to the arithmetic of the hyperelliptic curve $z^2 = f(x,y)$. Analogous results for non-hyperelliptic curves are also given.
\end{abstract}

\maketitle

\section{Introduction}

	A pair of symmetric bilinear forms $A$ and $B$ in $n$ variables over a field $k$ of characteristic not equal to $2$ determine a binary form
	\begin{equation*}
		f(x,y) := \disc(Ax-By) = (-1)^{\frac{n(n-1)}{2}}\det(Ax - By) :=  f_0x^n + f_1x^{n-1}y + \cdots + f_ny^n\,
	\end{equation*}
	of degree $n$. The question of whether a given binary form can be written as a discriminant form in this way is studied in \cite{BG, Wang, BGW_AIT2, BGW}. We prove that the property of being a discriminant form generically satisfies a local-global principle and deduce from this that most binary forms over $\Q$ are discriminant forms.
	
	It is easy to see that a binary form $f \in k[x,y]$ is a discriminant form if and only if the forms $c^2f$ are too, for every $c \in k^\times$. When $k = \Q$, it thus suffices to consider integral binary forms, in which case we define the height of $f$ to be $H(f) := \max\{|f_i|\}$ and consider the finite sets $N_n(X)$ of integral binary forms of degree $n$ with $H(f) < X$. We prove:
	
	\begin{Theorem}\label{thm:1}
		For any $n \ge 3$,
		\[
			\liminf_{X \to \infty}\frac{\#\{ f \in N_n(X) \;:\; \text{$f$ is a discriminant form over $\Q$\}}}{\#N_n(X)} > 75\%\,.
		\]
		Moreover, these values tend to $100\%$ as $n \to \infty$.
	\end{Theorem}

	When $n$ is even the corresponding $\limsup$ is strictly less than $100\%$, as can been seen by local considerations. For example, a square free binary form over $\R$ is a discriminant form if and only if it is not negative definite (see \cite[Section 7.2]{BGW}), a property which holds for a positive proportion of binary forms over $\R$. One is thus led to ask whether local obstructions are the only ones. This question was posed in \cite[Question 7.2]{BruinStoll}, albeit using somewhat different language. When $n = 2$, the answer is yes and turns out to be equivalent to the Hasse principle for conics and in this case the limit appearing in Theorem~\ref{thm:1} is the probability that a random conic has a rational point, which is $0$ (see \cite[Theorem 1.4]{BCF2}).
	
	  When $k$ is a number field, define the height of $f$ to be the height of the point $(f_0:\cdots:f_n)$ in weighted projective space $\PP^n(2:\cdots:2)$, and set $N_{n,k}(X)$ to be the finite set of degree $n$ binary forms over $k$ of height at most $X$. We prove the following:
	
	\begin{Theorem}\label{thm:2}
		Let $k$ be a number field. For any $n \ge 1$, 
		\[
			\lim_{X \to \infty}\frac{\#\{ f \in N_{n,k}(X) \;:\; \text{$f$ is a discriminant form over $k$}\}}
			{\#\{ f \in N_{n,k}(X) \;:\; \text{$f$ is a discriminant form everywhere locally}\}} = 100\%
		\]		
	\end{Theorem}
	
	It is known that a square free binary form $f(x,y)$ is a discriminant form over $k$ if the smooth projective hyperelliptic curve with affine model given by $z^2 = f(x,1)$ has a rational point \cite[Theorem 28]{BGW}. In particular binary forms of odd degree are discriminant forms. Results of Poonen and Stoll allow one to compute the proportion of hyperelliptic curves over $\Q$ of fixed genus that have points everywhere locally \cite{PS_ctpairing,PS_densities}. This gives lower bounds on the proportion of binary forms of fixed even degree that are locally discriminant forms. Computing these bounds and applying Theorem~\ref{thm:2}, one obtains Theorem~\ref{thm:1}.
	
	Theorem~\ref{thm:2} states that the property of being a discriminant form satisfies a local-global principle {\em generically}. This is rather surprising given that such a local-global principle does not hold {\em in general}. For example, there is a positive density set of positive square free integers $c$ such that the binary form
	\begin{equation}\label{eq:counterexample}
		f(x,y) = c(x^2 + y^2)(x^2 + 17y^2)(x^2-17y^2) \in \Q[x,y]
	\end{equation}
	is a discriminant form locally, but not over $\Q$ (see \cite[Theorem 11]{CreutzBLMS}). Of course, the forms appearing in~\eqref{eq:counterexample} are not generic. It is well known (as was first proved over $\Q$ by van der Waerden \cite{vdW}) that 100\% of degree $n$ univariate polynomials over a number field have Galois group $S_n$. Therefore Theorem~\ref{thm:2} is a consequence of the following:
	
	\begin{Theorem}\label{thm:3}
		Suppose $f(x,y) \in k[x,y]$ is a binary form of degree $n$ over a global field $k$ of characteristic not equal to $2$ and such that $f(x,1)$ has Galois group $S_n$. If $f(x,y)$ is a discriminant form everywhere locally, then $f(x,y)$ is a discriminant form over $k$.
	\end{Theorem}
	
	A square free binary form $f(x,y)$ of even degree gives an affine model of a smooth hyperelliptic curve $C:z^2=f(x,1)$ with two points at infinity. As shown in \cite{BGW} the $\SL_n(k)$-orbits of pairs $(A,B)$ with discriminant form $f(x,y)$ correspond to $k$-forms of the maximal abelian covering of $C$ of exponent $2$ unramified outside the pair of points at infinity. Geometrically these coverings arise as pullbacks of multiplication by $2$ on the generalized Jacobian $J_\frak{m}$ where $\frak{m}$ is the modulus comprising the points at infinity. The Galois-descent obstruction to the existence of such coverings over $k$ (and hence to the existence of a pencil of quadrics over $k$ with discriminant form $f(x,y)$) is an element of $\HH^2(k,J_\frak{m}[2])$ (\cite[Theorems 13 and 24]{BGW}). The Galois action on $J_\frak{m}[2]$ factors faithfully through the Galois group of $f(x,1)$, so Theorem \ref{thm:3} follows from:
	
	\begin{Theorem}\label{thm:4}
		Suppose $C$ is a hyperelliptic curve of genus $g$ over a global field $k$ of characteristic not equal to $2$ and that $\Gal(k(J_\frak{m}[2])/k) = S_{2g+2}$, then $\Sha^2(k,J_\frak{m}[2]) = 0$, i.e., an element of $\HH^2(k,J_\frak{m}[2])$ is trivial if it is everywhere locally trivial.
	\end{Theorem}
		
	This result is all the more surprising given that the analogous statement for the usual Jacobian is not true! There exist hyperelliptic curves of genus $g$ with Jacobian $J$, generic Galois action on $J[2]$ and such that $\Sha^1(k,J[2]) \simeq \Sha^2(k,J[2]) \ne 0$. A concrete example is given in \cite[Example 3.20]{PoonenRains}; see also Example~\ref{ex:more} below. This leads one to suspect that there may exist locally solvable hyperelliptic curves whose maximal abelian unramified covering of exponent $2$ does not descend to $k$ (or, equivalently, that the torsor $J^1$ parameterizing divisor classes of degree $1$ is not divisible by $2$ in the group $\HH^1(k,J)$). This can happen when the action of Galois on $J[2]$ is not generic; an example is given in \cite[Theorem 6.7]{CreutzViray}. But Theorem \ref{thm:4} implies that it cannot happen when the Galois action is generic:

	\begin{Theorem}
		Suppose $C$ is an everywhere locally soluble hyperelliptic curve satisfying the hypothesis of Theorem 4 and let $\Cbar$ denote the base change to a separable closure of $k$. Then
		\begin{itemize}
			\item[(a)] the maximal unramified abelian covering of $\Cbar$ of exponent $2$ descends to $k$, and
			\item[(b)] the maximal abelian covering of $\Cbar$ of exponent $2$ unramified outside $\frak{m}$ descends to $k$.			
		\end{itemize}
	\end{Theorem}
		\begin{proof}
			The covering in (a) is the maximal unramified subcovering of that in (b), while (b) follows from the Theorem \ref{thm:4} and the discussion preceding it.
		\end{proof}

	Theorem \ref{thm:4} generalizes to the context considered in \cite{CreutzBF}, which we now briefly summarize. Given a curve $C$, an integer $m$ and a reduced base point free effective divisor $\frak{m}$ on $C$, multiplication by $m$ on the generalized Jacobian $J_\frak{m}$ factors through an isogeny $\varphi:A_\frak{m}\to J_\frak{m}$ whose kernel is dual to the Galois module $\calJ[m] := (\Pic(C_\kbar)/\langle\frak{m}\rangle)[m]$. In the situation considered above $m = \deg(\frak{m}) = 2$ and $\varphi$ is multiplication by $2$ on $J_\frak{m}$ (in this case the duality is proved in \cite[Section 6]{PoonenSchaefer}). Via geometric class field theory the isogeny $\varphi$ corresponds to an abelian covering of $C_\kbar$ of exponent $m$ unramified outside $\frak{m}$. The maximal unramified subcoverings of the $k$-forms of this ramified covering are the $m$-coverings of $C$ parameterized by the explicit descents in \cite{BruinStoll,CreutzMathComp,BPS}. The Galois-descent obstruction to the existence of such a covering over $k$ is the class in $\HH^2(k,A_\frak{m}[\varphi])$ of the coboundary of $[J^1_\frak{m}]$ from the exact sequence $0 \to A_\frak{m}[\varphi] \to A_\frak{m} \to J_\frak{m} \to 0$.
	
	The following theorem says that the group $\Sha^2(k,A_\frak{m}[\varphi])$ is trivial provided the action of Galois on the $m$-torsion of the Jacobian is sufficiently generic. Theorem \ref{thm:4} is case (1).
	\begin{Theorem}\label{thm:5}
		Suppose $k$ is a global field of characteristic not dividing $m$, $C$ is a smooth projective and geometrically integral curve of genus $g$ over $k$, and let $m,\frak{m}$ and $\varphi$ be as above. In all of cases listed below, $\Sha^2(k,A_\frak{m}[\varphi]) = 0$.
		\begin{enumerate}
			\item $m = \deg(\frak{m}) = h^0([\frak{m}]) = 2$ and $\Gal(k(J_\frak{m}[2])/k) \simeq S_{2g+2}$.
			\item $m = 2$, $\frak{m}$ is a canonical divisor and $\Gal(k(J[2])/k) \simeq \Sp_{2g}(\F_2)$.
			\item $g=1$ and $m = \deg(\frak{m}) = 2$.
			\item $g=1$, $m = \deg(\frak{m}) = p^r$ for some prime $p$ and integer $r \ge 1$ and neither of the following hold:
				\begin{enumerate}
					\item The action of the absolute Galois group, $\Gal_k$, on $J[p]$ is reducible.
					\item The action of $\Gal_k$ on $J[p]$ factors through the symmetric group $S_3$.
				\end{enumerate}
		\end{enumerate}
	\end{Theorem}
	
	\begin{Remark}
		The statement and proof of the theorem depend only on the cohomology of the $\Gal_k$-module $\calJ[m] := (\Pic(C_\kbar)/\langle\frak{m}\rangle)[m]$ and its dual $A_\frak{m}[\varphi] = \calJ[m]^\vee$. If one likes, this can be taken as the definition of $A_\frak{m}[\varphi]$, and the isogeny can be ignored. 
	\end{Remark}
	
	In the case of genus one curves, the corresponding coverings can be described using the period-index obstruction map in \cite{CFOSS}. For example, a genus $1$ hyperelliptic curve $C:z^2 = f(x,y)$ can be made into a $2$-covering of its Jacobian. If $f(x,y)$ is the discriminant form of the pair $(A,B)$, then the quadric intersection $C':A=B=0$ in $\PP^3$ is a lift of $C$ to a $4$-covering of the Jacobian. This covering has trivial period-index obstruction in the sense described in \cite{CFOSS} and, conversely, any lift to a $4$-covering with trivial period-index obstruction may be given by an intersection of quadrics which generate a pencil with discriminant form $f(x,y)$. The analogous statement holds for any $m \ge 2$ (see \cite{CreutzBF}). Using this and Theorem \ref{thm:5} we obtain the following:
		
	\begin{Theorem}
		Fix $m \ge 2$. For 100\% of locally solvable genus one curves $C$ of degree $m$ there exists a genus one curve $D$ of period and index dividing $m^2$ such that $m[D] = [C]$ in the group $\HH^1(k,\Jac(C))$ paramterizing isomorphism classes of torsors under the Jacobian of $C$.
	\end{Theorem}

		When $m = 2$ case (3) of Theorem \ref{thm:5} shows that we may replace ``$100\%$'' with ``all''. This was first proved in the author's PhD thesis~\cite[Theorem 2.5]{CreutzPhD}. It would be interesting to determine if this is always true when $m$ is prime. In this case it is known that there always exists $D$ such that $m[D] = [C]$ \cite[Section 5]{Cassels} (but not for composite $m$ \cite{CreutzBLMS}). However, it is uknown whether $D$ may be chosen to have index dividing $m^2$.
	
	The proportion of locally solvable genus one curves of degree $3$ has been computed by Bhargava-Cremona-Fisher \cite{BCF}. As $100\%$ of cubic curves satisfy the hypothesis in case (4) of Theorem~\ref{thm:5}, this yields the following:
	
\begin{Theorem}
	At least 97\% of cubic curves $C$ admit a lift to a genus one curve $D$ of period and index $9$ such that $3[D]=[C]$ in the Weil-Ch\^atelet group of the Jacobian.
\end{Theorem}

\section{Proof of Theorem \ref{thm:5}}

		For a $\Gal_k$-module $M$ let
			\[
				\Sha^i(k,M) := \ker\left(\HH^i(k,M) \stackrel{\prod \res_v}\To \prod_{v}\HH^i(k_v,M)\right)\,,
			\]
		the product running over all completions of $k$. For a finite group $G$ and $G$-module $M$ define
			\[
				\HH_*^i(G,M) := \ker\left(\HH^i(G,M) \stackrel{\prod \res_g}\To \prod_{g \in G}\HH^i(\langle g \rangle,M)\right)\,.
			\]

		\begin{Lemma}
			\label{lem:vanishSha}
			Suppose $M$ is a finite $\Gal_k$-module and let $G := \Gal(k(M)/k)$ be the Galois group of its splitting field over $k$. Then
			\begin{enumerate}
				\item\label{lem:vanishSha1} $\Sha^1(k,M)$ is contained in the image of 
			$\HH_*^1(G,M)$ under the inflation map,
				\item\label{lem:vanishSha2} if $\HH^1_*(G,M) = 0$, then $\Sha^1(k,M)  = 0$, and
				\item\label{lem:vanishSha3} if $\HH^1_*(G,M^\vee) = 0$, then $\Sha^2(k,M)  = 0$.
			\end{enumerate}
		\end{Lemma}

		\begin{proof}
			(1) $\Rightarrow$ (2) because the inflation map is injective and (2) $\Rightarrow$ (3) by Tate's global duality theorem. We prove (1) using Chebotarev's density theorem as follows.

			Let $K = k(M)$ and for each place $v$ of $k$, choose a place $\frak{v}$ of $K$ above $v$ and let $G_{\frak{v}} = \Gal(K_{\frak{v}}/k_v)$ be the decomposition group. The inflation-restriction sequence gives the following commutative and exact diagram.
			\[
				\xymatrix{
					0 \ar[r]& \HH^1(G,M) \ar[r]^\inf\ar[d]^a & \HH^1(k,M) \ar[r]^\res \ar[d]^b & \HH^1(K,M) \ar[d]^c\\
					0 \ar[r]& \prod_{v}\HH^1(G_v,M) \ar[r] & \prod_v\HH^1(k_v,M) \ar[r] & \prod_v \HH^1(K_{\frak{v}},M) 
				}
			\]
			Since $M$ splits over $K$, we have $\HH^1(K,M) = \Hom_{cont}(\Gal_K,M)$. The map $c$ is therefore injective by Chebotarev's density theorem. Hence $\ker(b) = \inf(\ker(a))$. By a second application of Chebotarev's density theorem, the groups $G_v$ range (up to conjugacy) over all cyclic subgroups of $G$. From this it follows that $\ker(a) \subset \HH^1_*(G,M)$.
		\end{proof}
		
		Recall that $\calJ[m] := (\Pic(C_\kbar)/\frak{\langle m\rangle})[m]$. Since $\calJ[m]^\vee = A_\frak{m}[\varphi]$ (see \cite[]{CreutzBF}) it suffices to prove, under the hypothesis of Theorem \ref{thm:5}, that $\HH^1_*(\Gal(k(\calJ[m])/k),\calJ[m]) = 0$.

\subsection{Proof of Theorem~\ref{thm:5} case (1)}

	By assumption, the complete linear system associated to $\frak{m}$ gives a double cover of $\pi:C \to \PP^1$ which is not ramified at $\frak{m}$. Changing coordinates if necessary, we may arrange that $\frak{m}$ is the divisor above $\infty \in \PP^1$. Then $C$ is the hyperelliptic curve given by $z^2 = f(x,y)$, where $f(x,y)$ is a binary form of degree $n:=2g+2$ with nonzero discriminant. The ramification points of $\pi$ form a finite \'etale subscheme $\Delta\subset C$ of size $n$ which may be identified with the set of roots of $f(x,1)$.
	
	As described in \cite[Section 5]{PoonenSchaefer} (see also \cite[Proposition 22]{BGW}), we may identify $A_\frak{m}[\varphi] = J_\frak{m}[2] \simeq \Res^1_{\Delta}\mu_2$ with the subsets of $\Delta$ of even parity, while $\calJ[2] \simeq \Res_{\Delta}\mu_2/\mu_2$ corresponds to subsets modulo complements and $J[2] \simeq \Res^1_{\Delta}\mu_2/\mu_2$ corresponds to even subsets modulo complements. Parity of intersection defines a Galois equivariant and nondegenerate pairing,
	\[
		e:J_\frak{m}[2] \times \calJ[2] \to \Z/2\Z\,.
	\]
	(See \cite[Section 6]{PoonenSchaefer} or \cite{CreutzBF}). The induced pairing on $J[2]\times J[2]$ is the Weil pairing (written additively). Fixing an identification of the roots of $f(x,1)$ with the set $\{1,\dots,n\}$, the action of $\Gal_k$ on $\calJ[2]$ factors through the symmetric group $S_n$. The following lemma proves Theorem~\ref{thm:5}(1).

		\begin{Lemma}\label{lem:SncalJ2}
			$\HH^1_*(S_n,\calJ[2]) = 0$
		\end{Lemma}
	
		\begin{proof}
			For $t=1,\dots,n-1$, let $\tau_t$ denote the transposition $\tau_t := (t,t+1) \in S_n$ and let $P_t := \{t,t+1\} \in J_\fm[2]$ (recall $J_\fm[2]$ is identified with the even subsets of $\{1,\dots,n\}$). We use $\tilde{P}_t$ to denote the image of $P_t$ in $J[2] \subset \calJ[2]$. We note that for any $Q \in \calJ[2]$,
\[
	\tau_{t}(Q) + Q = e(P_t,Q)\tilde{P}_t.
\]
This is because $\tau_t$ is a transposition, addition is given by the symmetric difference, and the pairing $e$ is given by parity of intersection.

			Now suppose $\xi$ is a $1$-cocycle in $Z^1(S_n,\calJ[2])$ which represents a class in $\HH^1_*(S_n,\calJ[2])$. By our assumption, the restriction of $\xi$ to the subgroup $\langle \tau_t \rangle$ is a coboundary. Hence there is some $Q_t \in \calJ[2]$ such that $\xi_{\tau_t} = \tau_{t}(Q_t) + Q_t = e(P_t,Q)\tilde{P}_t\,.$ Since $P_1,\dots,P_{n-1}$ form a basis for $J_\fm[2]$ and $e$ is nondegenerate, we can find $Q \in \calJ[2]$ such that $e(P_t,Q) = e(P_t,Q_t)$ for all $t$. From this it follows that $\xi_{\tau_t} = \tau_t(Q) + Q$, for all $t$. In other words, $Q$ simultaneously plays witness to the fact that $\xi$ is a coboundary on each of the subgroups $\langle \tau_t \rangle$. But then $\xi$ must be a coboundary, since the $\tau_t$ generate $S_n$.
		\end{proof}

		\subsection{A lemma}
		
			Identifying $J[m]$ with $\Pic^0(C_\kbar)[m]$ gives an exact sequence,
		\[
			0 \to J[m] \stackrel{\iota}\To \calJ[m] \stackrel{\frac{1}{\ell} \deg}\To \Z/m\Z \to 0\,,
		\]
		where the integer $\ell$ is $\deg(\frak{m})/m$.
	
			\begin{Lemma}
				\label{lem:H1Ga}
				Let $G = \Gal(k(J[m])/k)$ and $G' = \Gal(k(\calJ[m])/k)$. The map $\iota_*\circ \inf : \HH^1(G,J[m]) \to \HH^1(G',\calJ[m])$ induces a surjection
				\[
					\ker\left( \HH^1(G,J[m]) \to \bigoplus_{g \in G'} \HH^1(\langle g\rangle, \calJ[m]) \right) \To \HH^1_*(G',\calJ[m])\,.
				\]
			\end{Lemma}

			\begin{proof}
				Cohomology of $G'$-modules gives an exact sequence
				\[
					\Z/m\Z \stackrel{\delta}\to \HH^1(G',J[m]) \to \HH^1(G',\calJ[m]) \to \HH^1(G',\Z/m\Z)\,.
				\]
				Since $\HH^1_*(G',\Z/m\Z) = 0$, we see that $\HH^1_*(G',\calJ[m])$ is contained in the image of $\HH^1(G',J[m])$. Hence, we may lift any $x \in \HH^1_*(G',\calJ[m])$ to some $y \in \HH^1(G',J[m])$. Now $G'$ sits in an exact sequence $0 \to N \to G' \to G \to 1$. We must show it is possible to choose $y$ such that $\res_N(y) = 0$, where $\res_N$ is as in the inflation-restriction sequence, 
				\[
					0 \to \HH^1(G,J[m]) \stackrel{\inf_N}\to \HH^1(G',J[m]) \stackrel{\res_N}{\to} \HH^1(N,J[m])^G\,.
				\]

				To that end we will determine $\res_N\circ\delta(1)$. Since $N$ acts trivially on $J[m]$ we have that $\HH^1(N,J[m])^G = \Hom_G(N,J[m])$ (here $G$ acts on the abelian group $N$ by conjugation in $G'$). Let $\epsilon \in \calJ[m]$ be a lift of $1 \in \Z/m\Z$. By definition, $\res_N\circ \delta(1)$ is (represented by) the map $i : N \to J[m]$ given by $i(\sigma) = \sigma(\epsilon)-\epsilon$. It follows from the general theory that $i$ is a morphism of $G$-modules. This is verified by the following computation:
				\begin{align*}
					i(\act{g} \sigma) &= i(\tilde{g}\sigma \tilde{g}^{-1}) & \text{(where $\tilde{g}$ is a lift of $g$ to $G'$)}\\
					&= \tilde{g}\left(\sigma(\tilde{g}^{-1}(\epsilon)) - \tilde{g}^{-1}(\epsilon)\right)\\
					&= \tilde{g}\left( \sigma(\epsilon+a) - (\epsilon + a)\right)  & \quad \text{(where $\tilde{g}^{-1}(\epsilon)-\epsilon = a \in J[m]$)}\\
					&= g\left( \sigma(\epsilon) - \epsilon\right) & \text{(since $N$ acts trivially on $J[m]$)}\\
					&=g(i(\sigma))
				\end{align*}
				We claim moreover that $i$ is injective. Indeed, if $\sigma \in N$ acts trivially on $\epsilon$, then $\sigma$ acts trivially on all of $\calJ[m]$ (because $N$ acts trivially on $J[m]$, and $\calJ[m]$ is generated by $J[m]$ and $\epsilon$).
				
				By assumption on $x$, $\res_g(y)$ lies in the image of $\delta:\Z/m\Z \to \HH^1(\langle g \rangle, J[m])$, for every $g \in G'$. It follows that $\res_N(y)$ lies in the subgroup $\End_G(N) = \Hom_G(N,i(N)) \subset \Hom_G(N,J[m])$ and, moreover, that this endomorphism sends every element to some multiple of itself. Any such endomorphism is a multiple of the identity, in which case $\res_N(y)$ is a multiple of $\res_N\circ \delta(1)$. We may therefore adjust our lift $y$ of $x$ by a multiple $\delta(1)$ to arrange that $\res_N(y) = 0$. This proves the lemma.
			\end{proof}

	\subsection{Proof of Theorem~\ref{thm:5} case (2)}

		Suppose $C$ has genus $g \ge 2$, $m = 2$ and $\frak{m}$ is a reduced and effective canonical divisor. Then $\deg(\frak{m}) = 2g-2$. Since the Weil pairing on $J[2]\times J[2]$ is alternating and Galois equivariant, the action of $\Gal_k$ on $J[2]$ factors through the symplectic group $\Sp(J[2]) \simeq \Sp_{2g}(\F_2)$. By \cite[Proposition 5.4]{BPS} the action of $\Gal_k$ on $\calJ[2] = \Pic(C_\kbar)/\langle \frak{[m]}\rangle$ also factors through $\Sp(J[2])$. We will show that $\HH^1_*(\Sp(J[2]),\calJ[2]) = 0$, after which the theorem follows from Lemma~\ref{lem:vanishSha3}(\ref{lem:vanishSha3}). 
		
		Let $\delta$ be the coboundary in $\Sp(J[2])$-cohomology of the exact sequence $0 \to J[2] \to \calJ[2] \to \Z/2\Z \to 0$. The sequence is not split, so $\delta(1)$ is nonzero. A direct (but rather involved) computation of group cohomology shows that $\HH^1(\Sp(V),V)$ has $\F_2$-dimension $1$ for any symplectic space $V$ of dimension $\ge 4$ over $\F_2$ (See \cite[Theorems 5.2, 4.8 and 4.1]{Pollatsek}). It follows that $\HH^1(\Sp(J[2]),J[2]) = \langle \delta(1) \rangle$. Since the image of $\delta(1)$ in $\HH^1(\Sp(J[2]),\calJ[2])$ is trivial, Lemma~\ref{lem:H1Ga} shows that $\HH^1_*(\Sp(J[2]),\calJ[2]) = 0$.

	\subsection{Proof of Theorem~\ref{thm:5} case (3)}		
		 We may assume $g=1$. Let $G = \Gal(k(J[2])/k)$. By Lemma~\ref{lem:H1Ga} it is enough to show that $\HH^1(G,J[2]) = 0$. Noting that $G \subset S_3$, let $G_0$ be the intersection of $G$ with the unique index $2$ subgroup of $S_3$. Then $G_0$ has odd order, so $\HH^1(G_0,J[2]) = 0$. Also, the order of $G/G_0$ and the characteristic of $J[2]^{G_0}$ both divide $2$, so $\HH^1(G/G_0,J[2]^{G_0})=0$. (This follows easily from the computation of cohomology of cyclic groups, since the conditions imply that the kernel of the norm is equal to the image of the augmentation ideal.) The inflation-restriction sequence then gives that $\HH^1(G,J[2]) = 0$ as desired.

	\subsection{Proof of Theorem~\ref{thm:5} case (4)}
		Suppose $C$ is a genus one curve and $\deg(\frak{m}) = m = p^r$ for some prime $p$ and positive integer $r$. It suffices to show that $\Sha^1(k,\calJ[m]) = 0$ except possibly when one of the following holds.
			\begin{enumerate}
				\item[(a)] The action of $\Gal_k$ on $J[p]$ is reducible, or
				\item[(b)] $r > 1$ and the action of $\Gal_k$ on $J[p]$ factors through the symmetric group $S_3$.
			\end{enumerate}
			So let us assume neither of these conditions holds. For $s \ge 1$ let $G_s = \Gal(k(J[p^s])/k)$. By Lemma~\ref{lem:H1Ga} it suffices to show that $\HH^1(G_r,J[p^r]) = 0$. The case $r = 1$ follows from Lemma~\ref{lem:caser=1} below. For $r > 1$,  \cite[Theorem 1]{CipStix} shows that the hypotheses ensure that the $G_1$-modules $J[p]$ and $\End(J[p])$ have no common irreducible subquotient. In this case the proof is completed by Lemma~\ref{lem:caser>1}.
		
		\begin{Lemma}\label{lem:caser=1}
			If $G_1$ acts on $J[p]$ irreducibly, then $\HH^1(G_1,J[p]) = 0$.
		\end{Lemma}

		\begin{proof}
			If $p \nmid \#G_1$, then $\HH^1(G_1,J[p])$ is obviously trivial. If $p \mid \#G_1$, then a well known result of Serre \cite[Proposition 15]{SerreEC} implies that either $\SL_2(\F_p) \subset G_1$ or $G_1$ is contained in a Borel subgroup of $\GL_2(\F_p)$. In the later case the action is reducible, so we may assume $\SL_2(\F_p) \subset G_1$. As the case $p = 2$ has already been addressed in the proof of case (3) of the theorem, we may assume $p$ is odd. In this case $G_1$ contains the normal subgroup $\mu_2$ of order prime to $p$ which has no fixed points. The corresponding inflation-restriction sequence,
			\[
				0 \to \HH^1(G/\mu_2,J[p]^{\mu_2}) \to \HH^1(G,J[p]) \to \HH^1(\mu_2,J[p])
			\]
			shows that $\HH^1(G_1,J[p])$ must vanish.
		\end{proof}
		
		\begin{Lemma}\label{lem:caser>1}
			If $\HH^1(G_1,J[p]) = 0$ and the $G_1$-modules $J[p]$ and $\End(J[p])$ have no common irreducible subquotient, then $\HH^1(G_r,J[p^r]) = 0$ for all $r \ge 1$. 
		\end{Lemma}
		
		\begin{proof}
			We will prove below that the hypothesis of the lemma implies that $\HH^1(G_r,J[p]) = 0$. Assuming this, induction on $s$ and the exact sequence $0 \to J[p^s] \to J[p^{s+1}] \to J[p] \to 0$ prove that $\HH^1(G_r,J[p^s]) = 0$ for every $1 \le s \le r$.
			
			As we have assumed $\HH^1(G_1,J[p]) = 0$, the inflation-restriction sequence gives an injective map $\HH^1(G_r,J[p]) \hookrightarrow \Hom_{G_1}(H_r,J[p])\,,$ where $H_r = \Gal(k(J[p^r])/k(J[p]))$ is the kernel of $G_r \to G_1$. The $G_1$-module $H_r$ admits a filtration $0 \subset H_1 \subset \dots \subset H_r$ whose successive quotients are $G_1$-submodules of $\ker\left(\GL(J[p^s]) \to \GL(J[p^{s-1}])\right) \cong \End(J[p])$. So the hypothesis of the lemma implies that $\Hom_{G_1}(H_r,J[p])=0$.
		\end{proof}

\section{Further Remarks}\label{sec:remarks}

 	The coboundary of $1$ under the exact sequence of $\Gal_k$-modules
 	\begin{equation}\label{eq:exact}
 		0 \to J[m] \to \calJ[m] \to \Z/m\Z \to 0
 	\end{equation}
 	gives a class in $\HH^1(k,J[m])$ which we denote by $[J^\ell_m]$. If either
 		 \begin{enumerate}
 		\item[(i)] $m = \deg(\frak{m}) = h^0([\frak{m}]) = 2$ and $g$ is even (i.e. $C$ is a hyperelliptic curve of even genus), or
 		\item[(ii)] $m = 2$ and $\frak{m}$ is a canonical divisor,
 	\end{enumerate}
 	then $[J^\ell_m] \in \HH^1(k,J[2])$ is the class of the theta characteristic torsor (cf. \cite[Definition 3.15]{PoonenRains}).

	\begin{Lemma}\label{lem:more}
		 The inclusion $J[m] \hookrightarrow \calJ[m]$ induces a map $\HH^1(k,J[m]) \to \HH^1(k,\calJ[m])$ for which the following hold.
		\begin{enumerate}
			\item $\Sha^1(k,\calJ[m])$ is contained in the image of $\HH^1(k,J[m])$.
			\item If $\Sha^1(k,\calJ[m])=0$, then $\Sha^1(k,J[m]) \subset \langle [J_m^\ell]\rangle$.
			\item If $[J^\ell_m] \in \Sha^1(k,J[m])$, then there is an exact sequence
				\[
					0 \to \langle [J^\ell_m] \rangle \subset \Sha^1(k,J[m]) \to \Sha^1(k,\calJ[m]) \to 0\,.
				\]
		\end{enumerate}
	\end{Lemma}

		\begin{proof}
			Take Galois cohomology of~\eqref{eq:exact} and use the fact that $\Sha^1(k,\Z/m\Z) = 0$. 
		\end{proof}
		
		\begin{Example}\label{ex:more}
			Suppose $C : y^2 = x^6+x+6$. Theorem~\ref{thm:5}, Lemma~\ref{lem:more} and \cite[Example 3.20b]{PoonenRains} together show that $\Sha^1(k,J[2]) = \langle [J^1_2] \rangle \simeq \Z/2\Z$,  while $\Sha^1(k,\calJ[2]) = 0$.
		\end{Example}
		
		In a sense this is the generic situation. Specifically one has the following.
		
		\begin{Proposition}
			Suppose that $m = 2$ and either
			\begin{enumerate}
				\item[(i)] $\deg(\frak{m}) = h^0([\frak{m}]) = 2$ and $\Gal(k(J[2])/k) \simeq S_{2g+2}$, or
			\item[(ii)] $\frak{m}$ is a canonical divisor and $\Gal(k(J[2])/k) \simeq \Sp_{2g}(\F_2)$.
			\end{enumerate}
			
			Let $S$ be any finite set of primes of $k$ containing all all archimedean primes, all primes above $2$,\  and all primes where $[J^\ell_2]$ is ramified. Then
			\[
				\Sha^1(k,S,J[2]) = \langle [J^\ell_2] \rangle \quad \text{and} \quad \Sha^1(k,S,\calJ[2]) = 0\,,
			\]
			where $\Sha^1(k,S,M) \subset \HH^1(k,M)$ denotes the subgroup that is locally trivial outside $S$.
		\end{Proposition}

		\begin{proof}
			We first note that the assumptions in case (i) require that the genus be even, so that in both cases $[J_2^\ell]$ is the theta characteristic torsor. It follows from \cite[Proposition 3.12]{PoonenRains} that $[J_2^\ell] \in \Sha^1(k,S,J[2])$. On the other hand, Lemma~\ref{lem:vanishSha} parts (1) and (2) and Lemma~\ref{lem:more} all remain valid if we replace $\Sha^1(k,M)$ with $\Sha^1(k,S,M)$.
		\end{proof}


\end{document}